\documentclass[12pt,american]{article}
\usepackage[T1]{fontenc}
\usepackage[latin9]{inputenc}
\usepackage{babel}
\usepackage{amsmath}
\usepackage{amsthm}
\usepackage{amssymb}
\usepackage[unicode=true,
 bookmarks=true,bookmarksnumbered=true,bookmarksopen=false,
 breaklinks=true,pdfborder={0 0 1},backref=false,colorlinks=true]
 {hyperref}
\hypersetup{pdftitle={Random Time Dynamical Systems I: General Structures},
 pdfauthor={Jos{\'e}L.~da Silva and Yuri Kondratiev},
 pdfsubject={Random Time Dynamical Systems},
 pdfkeywords={Dynamical Systems, random time change, subordinator, long time behavior},
 linkcolor=blue,citecolor=black,urlcolor=black,filecolor=black}

\makeatletter
\usepackage{soul}
\theoremstyle{plain}
\newtheorem{thm}{\protect\theoremname}[section]
\theoremstyle{definition}
\newtheorem{prop}{\protect\propositionname}[section]
\newtheorem{example}[thm]{\protect\examplename}
\theoremstyle{remark}
\newtheorem{rem}[thm]{\protect\remarkname}
\theoremstyle{plain}
\newtheorem{lem}[thm]{\protect\lemmaname}

\usepackage{babel}
\usepackage{units}
\usepackage{soul}
\usepackage{textcomp}
\usepackage{mathrsfs}
\usepackage{bbm}
\usepackage{pdfsync}

\usepackage{datetime}\usepackage{currfile}

\numberwithin{equation}{section}

\newcommand{\R}{{\mathbb R}}
\newcommand{\X}{{\R^d}}

\newcommand{\E}{{\mathbb E}}

\newcommand{\la}{\lambda}

\makeatother

\providecommand{\propositionname}{Proposition}
\providecommand{\examplename}{Example}
\providecommand{\lemmaname}{Lemma}
\providecommand{\remarkname}{Remark}
\providecommand{\theoremname}{Theorem}

\begin{document}
\title{Random Time Dynamical Systems I: General Structures}
\author{\textbf{Yuri Kondratiev}\\
 Department of Mathematics, University of Bielefeld,\\
 D-33615 Bielefeld, Germany,\\
 Dragomanov University, Kiev, Ukraine\\
 Email: kondrat@math.uni-bielefeld.de\and\textbf{Jos{\'e} Lu{\'i}s
da Silva},\\
 CIMA, University of Madeira, Campus da Penteada,\\
 9020-105 Funchal, Portugal.\\
 Email: joses@staff.uma.pt}
\date{\today}

\maketitle

\begin{abstract}
 In this paper we introduce the concept of  random time changes in  dynamical systems. 
 The subordination
principle may be applied 
to study the long time behavior of the random time systems. We show,
under certain assumptions on the class of random time, that the subordinated
system exhibits a slower time decay which is determined by the random
time characteristics.  In the path asymptotic a random time change is reflected in
the new velocity of the resulting dynamics. 


\end{abstract}

\section{Introduction}

The idea to consider stochastic processes with general random times
is known at least from the classical book by Gikhman and Skorokhod
\cite{GS74}. In the case of Markov processes time changes by subordinators
was considered already by Bochner \cite{Bochner1962}. In this case
it gives again a Markov process, so-called Bochner subordinated Markov
process. A more interesting situation appears in the case of inverse
subordinators. After the time change we do not obtain anymore a Markov
process and the study of such kind of processes become more challeging. 
At this point we would like to point out the  work
of Montroll and Weiss, 1965 \cite{Montroll1965} which considered
the physically motivated case of random walks in the random time.
This work created a wide research area related to the study of Markov
processes with inverse stable subordinators as random times changes,
see the book \cite{Meerschaert2012} for a detailed review and historical
comments.

For processes with random time change which are not subordinators
or inverse subordinators the situation is much less investigated. 
On the one hand additional assumptions on the stable subordinator turns 
the time change process very restrictive in applications. 
On the other hand, we find technical difficulties in
handling general inverse subordinators. These difficulties may be
overcome for certain sub-classes of inverse subordinators, see, e.g.,
\cite{KKdS19,KKS2018}.
The random time change is essential in modelling 
several physical systems (ecological and biological models), see, e.g., 
\cite{Magdziarz2015} and references therein for other applications.
A very particular choice of a random time process has not a special
motivation. 

There is a natural question concerning the use of a random
time change not only in stochastic dynamics but in a more wide class
of dynamical problems. In this paper we concentrate on the analysis
of this question in the case of dynamical systems in $\X$.

Let $X(t,x)$, $t\ge0$, $x\in\mathbb{R}^{d}$ be a dynamical system
in $\mathbb{R}^{d}$ starting from $x$, that is, $X(0,x)=x$. This
system is also a deterministic Markov process. Given $f:\mathbb{R}^{d}\longrightarrow\mathbb{R}$
we define 
\[
u(t,x):=f(X(t,x)).
\]
Then we have a version of the Kolmogorov equation, called the
Liouville equation in the theory of dynamical systems, namely 
\[
\frac{\partial}{\partial t}u(t,x)=Lu(t,x),
\]
where $L$ is the generator of a semigroup. This semigroup is given
by the solution of the Liouville equation, see for example \cite{EN2000,RS75,Yosida80}
for more details. 

If $E(t)$ is an inverse subordinator process (see Section \ref{sec:RT-FA} below for details and examples)
then we may consider the time changed random dynamical systems 
\[
Y(t):=X(E(t)).
\]
Our aim it to analyze the properties of $Y(t)$ depending on those
of the initial dynamical systems $X(t)$. In particular, we can
define 
\[
v(t,x):=\E[f(Y(t,x)]
\]
and try to compare the behavior $u(t,x)$ and $v(t,x)$ for a certain
class of functions $f$.

In the present paper we would like to present the main problems which
appear naturally in the study of random time changes in dynamical
systems. We illustrate the solutions to these problems with the simplest
examples is Section \ref{sec:RTD}. We postpone to a forthcoming
paper the study of particular classes of dynamical systems and random
times for which more detailed information may be obtained. In a certain
sense, in this paper foundations are presented to study random time changes
in dynamical systems. We hope that this program may be interesting
and attractive for mathematicians working and studying the interplay
between stochastic and dynamical systems theories. Having this in
mind, we will not elaborate detailed conditions on the dynamical systems
and random times in our considerations. The latter shall be (and may
be) done for each particular model we would like to study in detail.

The rest of the paper is organized as follows. In Section \ref{sec:RT-FA}
we present the class of inverse subordinators and the associated general
fractional derivatives. In Section \ref{sec:RTD} we consider the
simplest examples of dynamical systems and present the first results
when the random time is associated to the $\alpha$-space subordinator.
In Section \ref{sec:PGM} we consider a dynamical system as a deterministic
Markov processes and introduce the notion of potential and Green measure
of the dynamical system. Finally, in Section \ref{sec:PT} we investigate
the path transformation of a simple dynamical system by a random time.

\section{Random Times and Fractional Analysis}

\label{sec:RT-FA}In this section we introduce the inverse subordinators
and the corresponding general fractional derivatives. Associated to
these classes of inverse subordinators we define a kernel $k\in L_{\mathrm{loc}}^{1}(\mathbb{R}_{+})$
which is used to define a general fractional derivatives (GFD), see
\cite{Kochubei11} for details and applications to fractional differential
equations. These admissible kernels $k$ are characterized in terms
of their Laplace transforms $\mathcal{K}(\lambda)$ as $\lambda\to0$,
see assumption (H) below and also Lemma \ref{KL}.

Let $S=\{S(t),\;t\ge0\}$ be a subordinator without drift starting
at zero, that is, an increasing L{\'e}vy process starting at zero,
see \cite{Bertoin96} for more details. The Laplace transform of $S(t)$,
$t\ge0$ is expressed in terms of a Bernstein function $\Phi:[0,\infty)\longrightarrow[0,\infty)$
(also known as Laplace exponent) by 
\[
\mathbb{E}(e^{-\lambda S(t)})=e^{-t\Phi(\lambda)},\quad\lambda\ge0.
\]
The function $\Phi$ admits the representation 
\begin{equation}
\Phi(\lambda)=\int_{(0,\infty)}(1-e^{-\lambda\tau})\,\mathrm{d}\sigma(\tau),\label{eq:Levy-Khintchine}
\end{equation}
where the measure $\sigma$ (called L{\'e}vy measure) has support
in $[0,\infty)$ and fulfills 
\begin{equation}
\int_{(0,\infty)}(1\wedge\tau)\,\mathrm{d}\sigma(\tau)<\infty.\label{eq:Levy-condition}
\end{equation}
In what follows we assume that the L{\'e}vy measure $\sigma$ satisfy
\begin{equation}
\sigma\big((0,\infty)\big)=\infty.\label{eq:Levy-massumption}
\end{equation}
Using the L{\'e}vy measure $\sigma$ we define the kernel $k$ as
follows 
\begin{equation}
k:(0,\infty)\longrightarrow(0,\infty),\;t\mapsto k(t):=\sigma\big((t,\infty)\big).\label{eq:k}
\end{equation}
Its Laplace transform is denoted by $\mathcal{K}$, that is, for any
$\lambda\ge0$ one has 
\begin{equation}
\mathcal{K}(\lambda):=\int_{0}^{\infty}e^{-\lambda t}k(t)\,\mathrm{d}t.\label{eq:LT-k}
\end{equation}
The relation between the function $\mathcal{K}$ and the Laplace exponent
$\Phi$ is given by 
\begin{equation}
\Phi(\lambda)=\lambda\mathcal{K}(\lambda),\quad\forall\lambda\ge0.\label{eq:Laplace-exponent}
\end{equation}

In what follows we make the following assumption on the Laplace exponent
$\Phi(\lambda)$ of the subordinator $S$. 
\begin{description}
\item [{(H)}] $\Phi$ is a complete Bernstein function (that is, the L{\'e}vy
measure $\sigma$ is absolutely continuous with respect to the Lebesgue
measure) and the functions $\mathcal{K}$, $\Phi$ satisfy 
\begin{equation}
\mathcal{K}(\lambda)\to\infty,\text{ as \ensuremath{\lambda\to0}};\quad\mathcal{K}(\lambda)\to0,\text{ as \ensuremath{\lambda\to\infty}};\label{eq:H1}
\end{equation}
\begin{equation}
\Phi(\lambda)\to0,\text{ as \ensuremath{\lambda\to0}};\quad\Phi(\lambda)\to\infty,\text{ as \ensuremath{\lambda\to\infty}}.\label{eq:H2}
\end{equation}
\end{description}
\begin{example}[$\alpha$-stable subordinator]
\label{exa:alpha-stable1}A classical example of a subordinator $S$
is the so-called $\alpha$-stable process with index $\alpha\in(0,1)$.
Specifically, a subordinator is $\alpha$-stable if its Laplace exponent
is 
\[
\Phi(\lambda)=\lambda^{\alpha}=\frac{\alpha}{\Gamma(1-\alpha)}\int_{0}^{\infty}(1-e^{-\lambda\tau})\tau^{-1-\alpha}\,\mathrm{d}\tau.
\]
In this case it follows that the L{\'e}vy measure is $\mathrm{d}\sigma_{\alpha}(\tau)=\frac{\alpha}{\Gamma(1-\alpha)}\tau^{-(1+\alpha)}\,\mathrm{d}\tau$. 
The corresponding kernel $k_{\alpha}$ has the form $k_{\alpha}(t)=g_{1-\alpha}(t):=\frac{t^{-\alpha}}{\Gamma(1-\alpha)}$,
$t\ge0$ and its Laplace transform is $\mathcal{K}_{\alpha}(\lambda)=\lambda^{\alpha-1}$,
$\lambda\ge0$. 
\end{example}

\begin{example}[Gamma subordinator]
\label{exa:gamma-subordinator}The Gamma process $Y^{(a,b)}$ with
parameters $a,b>0$ is another example of a subordinator with Laplace
exponent 
\[
\Phi_{(a,b)}(\lambda)=a\log\left(1+\frac{\lambda}{b}\right)=\int_{0}^{\infty}(1-e^{-\lambda\tau})a\tau^{-1}e^{-b\tau}\,\mathrm{d}\tau,
\]
the second equality is known as the Frullani integral. The L{\'e}vy
measure is given by $d\sigma_{(a,b)}(\tau)=a\tau^{-1}e^{-b\tau}\,\mathrm{d}\tau.$
The associated kernel $k_{(a,b)}(t)=a\Gamma(0,bt)$, $t>0$, (here $\Gamma(\nu,z):=\int_{z}^{\infty}e^{-t}t^{\nu-1}\,\mathrm{d}t$
is the incomplete gamma function, see Section 8.3 in \cite{GR81}) and its
Laplace transform is $\mathcal{K}_{(a,b)}(\lambda)=a\lambda^{-1}\log(1+\frac{\lambda}{b})$,
$\lambda>0$. 
\end{example}

\begin{example}[Truncated $\alpha$-stable subordinator]
\label{exa:truncated_stable_ sub}The truncated $\alpha$-stable
subordinator (see Example 2.1-(ii) in \cite{Chen2017}) $S_{\delta}$,
$\delta>0$ is a driftless $\alpha$-stable subordinator with L{\'e}vy
measure given by 
\[
\mathrm{d}\sigma_{\delta}(\tau):=\frac{\alpha}{\Gamma(1-\alpha)}\tau^{-(1+\alpha)}1\!\!1_{(0,\delta]}(\tau)\,\mathrm{d}\tau,\qquad\delta>0.
\]
The corresponding Laplace exponent is 
\[
\Phi_{\delta}(\lambda)=\lambda^{\alpha}\left(1-\frac{\Gamma(-\alpha,\delta\lambda)}{\Gamma(-\alpha)}\right)+\frac{\delta^{-\alpha}}{\Gamma(1-\alpha)}
\]
and the associated kernel $k_{\delta}$ is given by 
\[
k_{\delta}(t):=\sigma_{\delta}\big((t,\infty)\big)=\frac{1\!\!1_{(0,\delta]}(t)}{\Gamma(1-\beta)}(t^{-\beta}-\delta^{-\beta}),\;t>0.
\]
\end{example}

\begin{example}[Sum of two alpha stable subordinators]
\label{exa:sum-two-stables}Let $0<\alpha<\beta<1$ be given and
$S_{\alpha,\beta}(t)$, $t\ge0$ the driftless subordinator with Laplace
exponent given by 
\[
\Phi_{\alpha,\beta}(\lambda)=\lambda^{\alpha}+\lambda^{\beta}.
\]
It is clear from Example \ref{exa:alpha-stable1}  that the corresponding L{\'e}vy measure
$\sigma_{\alpha,\beta}$ is the sum of two L{\'e}vy measures, that
is, 
\[
\mathrm{d}\sigma_{\alpha,\beta}(\tau)=\mathrm{d}\sigma_{\alpha}(\tau)+\mathrm{d}\sigma_{\alpha}(\tau)=\frac{\alpha}{\Gamma(1-\alpha)}\tau^{-(1+\alpha)}\,\mathrm{d}\tau+\frac{\beta}{\Gamma(1-\beta)}\tau^{-(1+\beta)}\,\mathrm{d}\tau.
\]
Then the associated kernel $k_{\alpha,\beta}$ is 
\[
k_{\alpha,\beta}(t):=g_{1-\alpha}(t)+g_{1-\beta}(t)=\frac{t^{-\alpha}}{\Gamma(1-\alpha)}+\frac{t^{-\beta}}{\Gamma(1-\beta)},\;t>0
\]
and its Laplace transform is $\mathcal{K}_{\alpha,\beta}(\lambda)=\mathcal{K}_{\alpha}(\lambda)+\mathcal{K}_{\beta}(\lambda)=\lambda^{\alpha-1}+\lambda^{\beta-1}$,
$\lambda>0$. 
\end{example}

\begin{example}[Kernel with exponential weight]
\label{exa:exponential-weight} Given
$\gamma>0$ and $0<\alpha<1$ consider the subordinator with Laplace
exponent 
\[
\Phi_{\gamma}(\lambda):=(\lambda+\gamma)^{\alpha}=\left(\frac{\lambda+\gamma}{\lambda}\right)^{1+\alpha}\frac{\alpha}{\Gamma(1-\alpha)}\int_{0}^{\infty}(1-e^{-\lambda\tau})\tau^{-1-\alpha}\,\mathrm{d}\tau.
\]
It follows that the L{\'e}vy measure is given by $\mathrm{d}\sigma_{\gamma}(\tau)=\left(\frac{\lambda+\gamma}{\lambda}\right)^{1+\alpha}\frac{\alpha}{\Gamma(1-\alpha)}\tau^{-(1+\alpha)}\mathrm{d}\tau$
which yields a kernel $k_{\gamma}$ with exponential weight, namely
\[
k_{\gamma}(t)=g_{1-\alpha}(t)e^{-\gamma t}=\frac{t^{-\alpha}}{\Gamma(1-\alpha)}e^{-\gamma t}.
\]
The corresponding Laplace transform of $k_{\gamma}$ is given by $\mathcal{K}_{\gamma}(\lambda)=\lambda^{-1}(\lambda+\gamma)^{\alpha}$,
$\lambda>0$. 
\end{example}

\begin{rem} The subordinators from Examples \ref{exa:alpha-stable1}--\ref{exa:exponential-weight} 
provides different types of kernels $k$ which give rise to different types of fractional derivatives. 
These fractional derivatives were introduced to study the theory of relaxation and diffusion equations, see \cite{Kochubei11}
 and references therein.
\end{rem}

Denote by $E$ the inverse process of the subordinator $S$, that
is, 
\begin{equation}
E(t):=\inf\{s>0\mid S(s)>t\}=\sup\{s\ge0\mid S(s)\le t\}.\label{eq:inverse-sub}
\end{equation}
For any $t\ge0$ we denote by $G_{t}^{k}(\tau):=G_{t}(\tau)$, $\tau\ge0$
the marginal density of $E(t)$ or, equivalently 
\[
G_{t}(\tau)\,\mathrm{d}\tau=\frac{\partial}{\partial\tau}P(E(t)\le\tau)\,\mathrm{d}\tau=\frac{\partial}{\partial\tau}P(S(\tau)\ge t)\,\mathrm{d}\tau=-\frac{\partial}{\partial\tau}P(S(\tau)<t)\,\mathrm{d}\tau.
\]

As the density $G_{t}(\tau)$ plays an important role in what follows,
we collect the most important properties needed later on. 
\begin{rem}
\label{rem:distr-alphastab-E}If $S$ is the $\alpha$-stable process,
$\alpha\in(0,1)$, then the inverse process $E(t)$ has Laplace transform
(cf.~Prop.~1(a) in \cite{Bingham1971}) given by 
\begin{equation}
\mathbb{E}(e^{-\lambda E(t)})=\int_{0}^{\infty}e^{-\lambda\tau}G_{t}(\tau)\,\mathrm{d}\tau=\sum_{n=0}^{\infty}\frac{(-\lambda t^{\alpha})^{n}}{\Gamma(n\alpha+1)}=E_{\alpha}(-\lambda t^{\alpha}).\label{eq:Laplace-density-alpha}
\end{equation}
It follows from the asymptotic behavior of the Mittag-Leffler function
$E_{\alpha}$ that $\mathbb{E}(e^{-\lambda E(t)})\sim Ct^{-\alpha}$
as $t\to\infty$. Using the properties of the Mittag-Leffler function
$E_{\alpha}$, we can show that the density $G_{t}(\tau)$ is given
in terms of the Wright function $W_{\mu,\nu}$, namely $G_{t}(\tau)=t^{-\alpha}W_{-\alpha,1-\alpha}(\tau t^{-\alpha})$,
see \cite{Gorenflo1999} for more details. 
\end{rem}

For a general subordinator, the following lemma determines the $t$-Laplace
transform of $G_{t}(\tau)$, with $k$ and $\mathcal{K}$ given in
\eqref{eq:k} and \eqref{eq:LT-k}, respectively. For the proof see
\cite{Kochubei11} or Lemma~3.1 in \cite{Toaldo2015}. 
\begin{lem}
\label{lem:t-LT-G}The $t$-Laplace transform of the density $G_{t}(\tau)$
is given by 
\begin{equation}
\int_{0}^{\infty}e^{-\lambda t}G_{t}(\tau)\,\mathrm{d}t=\mathcal{K}(\lambda)e^{-\tau\lambda\mathcal{K}(\lambda)}.\label{eq:LT-G-t}
\end{equation}
The double ($\tau,t$)-Laplace transform of $G_{t}(\tau)$ is 
\begin{equation}
\int_{0}^{\infty}\int_{0}^{\infty}e^{-p\tau}e^{-\lambda t}G_{t}(\tau)\,\mathrm{d}t\,\mathrm{d}\tau=\frac{\mathcal{K}(\lambda)}{\lambda\mathcal{K}(\lambda)+p}.\label{eq:double-Laplace}
\end{equation}
\end{lem}

For any $\alpha\in(0,1)$ the Caputo-Dzhrbashyan fractional derivative
of order $\alpha$ of a function $u$ is defined by (see e.g., \cite{KST2006}
and references therein) 
\begin{equation}
\big(\mathbb{D}_{t}^{\alpha}u\big)(t)=\frac{d}{dt}\int_{0}^{t}k_{\alpha}(t-\tau)u(\tau)\,\mathrm{d}\tau-k_{\alpha}(t)u(0),\quad t>0,\label{eq:Caputo-derivative}
\end{equation}
where $k_{\alpha}$ is given in Example \ref{exa:alpha-stable1},
that is, $k_{\alpha}(t)=g_{1-\alpha}(t)=\frac{t^{-\alpha}}{\Gamma(1-\alpha)}$,
$t>0$. In general, starting with a subordinator $S$ and the kernel
$k\in L_{\mathrm{loc}}^{1}(\mathbb{R}_{+})$ given in \eqref{eq:k},
we may define a differential-convolution operator by 
\begin{equation}
\big(\mathbb{D}_{t}^{(k)}u\big)(t)=\frac{d}{dt}\int_{0}^{t}k(t-\tau)u(\tau)\,\mathrm{d}\tau-k(t)u(0),\;t>0.\label{eq:general-derivative}
\end{equation}
The operator $\mathbb{D}_{t}^{(k)}$ is also known as generalized
fractional derivative. 
\begin{example}[Distributed order derivative]
\label{exa:distr-order-deriv}Consider the kernel $k$ defined by
\begin{equation}
k(t):=\int_{0}^{1}g_{\alpha}(t)\,\mathrm{d}\alpha=\int_{0}^{1}\frac{t^{\alpha-1}}{\Gamma(\alpha)}\,\mathrm{d}\alpha,\quad t>0.\label{eq:distributed-kernel}
\end{equation}
Then it is easy to see that 
\[
\mathcal{K}(\lambda)=\int_{0}^{\infty}e^{-\lambda t}k(t)\,\mathrm{d}t=\frac{\lambda-1}{\lambda\log(\lambda)},\quad\lambda>0.
\]
The corresponding differential-convolution operator $\mathbb{D}_{t}^{(k)}$
is called distributed order derivative, see \cite{Atanackovic2009,Daftardar-Gejji2008,Hanyga2007,Kochubei2008,Gorenflo2005,Meerschaert2006}
for more details and applications. 
\end{example}

We say that the functions $f$ and $g$ are \emph{asymptotically equivalent
at infinity}, and denote $f\sim g$ as $x\to\infty$, meaning that
\[
\lim_{x\to\infty}\frac{f(x)}{g(x)}=1.
\]
We say that a function $f$ is slowly varying (see \cite{Feller71,Seneta1976})
if 
\[
\lim_{x\to\infty}\frac{f(\lambda x)}{f(x)}=1,\quad\mathrm{for\;any}\;\lambda>0.
\]

The following lemma may be extracted from the proof of Theorem 4.2
in \cite{KocKon2017} which is used below. 
\begin{lem}
\label{KL}Assume that the subordinator $S(t)$ and its inverse $E(t)$,
$t\ge0$ are such that 
\begin{equation}
\mathcal{K}(\la)\sim\la^{-\gamma}Q\left(\frac{1}{\la}\right),\quad\la\to0,\label{eq:condition-K}
\end{equation}
where $0\leq\gamma\leq1$ and $Q$ is a slowly varying function. In
addition define 
\[
A(t,z):=\int_{0}^{\infty}e^{-z\tau}G_{t}(\tau)d\tau,\quad t>0,\;z>0.
\]
Then 
\[
A(t,z)\sim\frac{1}{z}\frac{t^{\gamma-1}}{\Gamma(\gamma)}Q(t),\quad t\to\infty.
\]
\end{lem}

\begin{rem}
We point out that the condition \eqref{eq:condition-K} on the Laplace
transform of the kernel $k$ is satisfied by all Examples \ref{exa:alpha-stable1}--\ref{exa:exponential-weight}
and \ref{exa:distr-order-deriv} above. The case of Example \ref{exa:sum-two-stables}
is easily checked as 
\[
\mathcal{K}(\lambda)=\lambda^{\alpha}+\lambda^{\beta}=\lambda^{-(1-\alpha)}(1+\lambda^{-(\alpha-\beta)})=\lambda^{-\gamma}Q\left(\frac{1}{\lambda}\right),
\]
where $\gamma=1-\alpha>0$ and $Q(t)=1+t^{\alpha-\beta}$ is a slowly
varying function. 

\end{rem}

\section{Random Time Dynamics}

\label{sec:RTD}In this section we study the effect of the subordination
by the density $G_{t}(\tau)$ of the inverse process $E(t)$, $t\ge0$
of a dynamical system.

Define a random dynamical systems 
\[
Y(t,x,\omega):=X(E(t,\omega),x,\omega),\quad x\in\X.
\]
For suitable functions $f:\X\longrightarrow\R$ define 
\[
v(t,x):=\E[f(Y(t,x))].
\]
Then $v(t,x)$ is the solution to an evolution equation with the same
generator $L$ but with generalized fractional derivative (see \eqref{eq:general-derivative}),
namely 
\begin{equation}
\label{fde-for-sub-sol}
\mathbb{D}_{t}^{(k)}v(t,x)=Lv(t,x).
\end{equation}
As a result the following subordination formula holds: 
\begin{equation}
v(t,x)=\int_{0}^{\infty}u(\tau,x)G_{t}(\tau)\,\mathrm{d}\tau.\label{SUBFOR}
\end{equation}
The problem (as in the Markov case) is to study the change of the behavior
of $u$ after subordination. The formula (\ref{SUBFOR}) is the main
object for the analysis of this problem. 

In the following proposition we consider
the simplest evolution equation in $\X$
\[
\frac{\mathrm{d}}{\mathrm{d}t}X(t)=v\in\X,\quad X(0)=x_{0}\in\X.
\]
The corresponding dynamics is 
\[
X(t)=x_{0}+vt,\quad t\geq0.
\]
We assume $x_{0}=0$ for simplicity.
Take $f(x)=e^{-\alpha|x|},\;\alpha>0$.
The corresponding solution to the Liouville equation is 
\[
u(t,x)=e^{-\alpha t|v|},\quad t\geq0.
\]

\begin{prop}
 Assume that the assumptions
of Lemma \ref{KL} are satisfied. Then
\[
v(t,x)\sim\frac{1}{\alpha|v|\Gamma(\gamma)}t^{\gamma-1}Q(t),\quad t\to\infty.
\]

\end{prop}

\begin{proof}

From the explicit form of the solution $u(t,x)$
using (\ref{SUBFOR}) and Lemma \ref{KL} we obtain 
\[
v(t,x)\sim\frac{1}{\alpha|v|\Gamma(\gamma)}t^{\gamma-1}Q(t),\quad t\to\infty.
\]
In particular, for the $\alpha$-stable subordinator considered in
Example \ref{exa:alpha-stable1}, we obtain $v(t,x)\sim Ct^{-\alpha}$,
$C$ is a constant. Therefore, starting with a solution $u(t,x)$
with exponential decay after subordination we observe a polynomial
decay with the order defined by the random time characteristics. 
\end{proof}

\begin{example}
\label{exa:DS-2}For $d=1$ consider the dynamics
\[
\beta\frac{\mathrm{d}}{\mathrm{d}t}X(t)=\frac{1}{X^{\beta-1}(t)},\quad\beta\geq1.
\]
It is clear that the solution is given by 
\[
X(t)=(t+C)^{1/\beta}.
\]
Take the function $f(x)=\exp(-a|x|^{\beta})$, $a>0$, then the long
time behavior of the subordination $v(t,x)$ is given by 
\[
v(t,x)\sim\frac{e^{-aC}}{a}\frac{t^{\gamma-1}}{\Gamma(\gamma)}Q(t),\quad t\to\infty.
\]
In particular if we choose the density $G_{t}(\tau)$ of the inverse
subordinator $E(t)$ corresponding to the Example \ref{exa:sum-two-stables},
then we obtain
\[
v(t,x)\sim Ct^{-\alpha}(1+t^{\alpha-\beta})\sim Ct^{-\alpha},\quad t\to\infty.
\]
\end{example}

\section{Potentials and Green Measures}

\label{sec:PGM}The notion of potential is classical in the theory
of Markov processes, see, e.g., \cite{Blumenthal1968}. Recently it
was proposed the concept of Green measure as a representation of potentials
in an integral form, see \cite{KdS2020}. The modification of these
concepts for time changed Markov processes was investigated in \cite{KdS20}

Considering a dynamical system as a deterministic Markov processes,
we have the possibility to study the notion of potential and Green measure in this context.

In our framework above for a function $f:\X\to\R$ consider the solution
to the Cauchy problem 
\[
\frac{\partial}{\partial t}u(t,x)=Lu(t,x),
\]
\[
u(0,x)=f(x).
\]

Then we obtain 
\[
u(t,x)=(e^{tL}f)(x).
\]
Define the potential for the function $f$ by 
\[
U(f,x):=\int_{0}^{\infty}u(t,x)\,\mathrm{d}t=\int_{0}^{\infty}(e^{tL}f)(x)\,\mathrm{d}t=-(L^{-1}f)(x), \quad x\in\X.
\]
The existence of $U(f,x)$ is not clear at all. It depends on the
class of functions $f$ and the Liouville generator $L$. Assuming
the existence of $U(f,x)$ we would like to obtain an integral representation 
\begin{equation}
\label{eq:potential-U}
U(f,x)=\int_{\X}f(y)\,\mathrm{d}\mu^{x}(y)
\end{equation}
with a Radom measure $\mu^{x}$ on $\X$. This measure we will call
the Green measure for our dynamical system. As in the case of Markov
processes, the definition of the potential is easy to introduce but
difficult to analyze for each particular model.

Moreover,  on the base of certain particular examples we 
may assume that the potentials are well  defined for very
special classes  functions $f$. But the existence of Green
measure  we can not expect. This point we will discuss in details in 
a forthcoming paper.

After a random time change we will have the subordinated solution
$v(t,x)$ for the fractional equation (see equation \eqref{fde-for-sub-sol} above). Then we can try
again to define the potential 
\[
V(f,x):=\int_{0}^{\infty}v(t,x)\,\mathrm{d}t, \quad x\in\X.
\]
But it is not hard to see that for general random times this integral
will be divergent. In fact, it follows from the subordination formula \eqref{SUBFOR},
and the Fubini theorem that 
\[
V(f,x)=\int_{0}^{\infty}\int_{0}^{\infty}u(\tau,x)G_{t}(\tau)\,\mathrm{d}\tau\,\mathrm{d}t
=\int_{0}^{\infty}u(\tau,x)\left(\int_{0}^{\infty}G_{t}(\tau)\,\mathrm{d}t\right)\mathrm{d}t\,\mathrm{d}\tau.
\]
The inner integral is not convergent due to equality \eqref{eq:LT-G-t} and assumption \eqref{eq:H1}.
In view to overcome this difficulty we may use the notion of renormalized potential. 
More precisely, inspired by the time change of Markov processes (see \cite{KdS20} for details), we define
the renormalized potential 
\begin{equation}
V_r(f,x):=\lim_{t\to\infty}\frac{1}{N(t)}\int_0^tv(s,x)\,\mathrm{d}s, \quad t\ge0,
\end{equation}
where  $N(t)$ is defined by $N(t):=\int_0^tk(s)\,\mathrm{d}s$.
 Then assuming the existence of $U(f,x)$ it is not difficult to show that 
\[
V_r(f,x)=\int_{0}^{\infty}u(t,x)\,\mathrm{d}t.
\] 

\section{Path Transformation}

\label{sec:PT}Now we investigate the transformation of the trajectories
of dynamical systems under random times. As above we have the Liouville
equation for 
\[
u(t,x):=f(X(t,x)),\quad t\ge0,\;x\in\X,
\]
that is, 
\[
\frac{\partial}{\partial t}u(t,x)=Lu(t,x),\quad u(0,x)=f(x),
\]
where $L$ is the generator of a semigroup. In addition, let $E(t)$,
$t\ge0$ be the inverse subordinator process, then we can consider
the time changed random dynamical systems 
\[
Y(t,x)=X(E(t),x),\quad t\ge0,\;x\in\X.
\]
Define 
\[
v(t,x):=\E[f(Y(t,x)].
\]
The subordination formula gives 
\[
v(t,x)=\int_{0}^{\infty}u(\tau,x)G_{t}(\tau)\,\mathrm{d}\tau.
\]

Now we will take the vector-function 
\[
f(x)=x\in\X.
\]
Then the average trajectories of $Y(t,x)$ is given by

\[
\E(Y(t,x))= \int_{0}^{\infty}X(\tau,x)G_{t}(\tau)\,\mathrm{d}\tau.
\]
\begin{example}
If we consider the dynamical system of Proposition 3.1, that
is, $X(t,x)=vt$, then we obtain 
\[
\mathbb{E}[Y(t,x)]=v\int_{0}^{\infty}\tau  G_{t}(\tau)\,\mathrm{d}\tau.
\]

Therefore, we need to know the first moment of the density $G_t$.
Consider the case of the inverse $\alpha$-stable subordinator from Example 2.1. Then
$$
\int_0^\infty \tau G_t(\tau) d\tau = C t^{\alpha}.
$$
Therefore, the asymptotic of the time changed trajectory will be
slower (proportional to $t^{\alpha}$) instead of initial linear $vt$ motion.
In a forthcoming paper we will study in detail these results  for other classes 
of inverse subordinators. 

\end{example}

\subsection*{Acknowledgments}

This work has been partially supported by Center for Research in Mathematics
and Applications (CIMA) related with the Statistics, Stochastic Processes
and Applications (SSPA) group, through the grant UIDB/MAT/04674/2020
of FCT-Funda{\c c\~a}o para a Ci{\^e}ncia e a Tecnologia, Portugal.

The financial support by the Ministry for Science and Education of
Ukraine through Project 0119U002583 is gratefully acknowledged.

\end{document}